\pdfoutput=1
\documentclass{amsart}

\usepackage[utf8]{inputenc}
\usepackage{color}
\usepackage[dvipsnames]{xcolor}
\usepackage{hyperref}
\usepackage{amssymb}
\usepackage{float}
\usepackage[nameinlink, capitalize, noabbrev]{cleveref}
\usepackage{tikz}
\usepackage{forest}
\usepackage{comment}
\usepackage{nicefrac}
\usepackage{enumerate}
\usepackage{booktabs}
\usetikzlibrary{shapes.geometric,decorations.markings}
\usetikzlibrary{decorations.pathreplacing}
\usetikzlibrary{fit}
\usetikzlibrary{positioning}
\tikzset{mytext/.style={font=\small, text=black}}
\hypersetup{
    colorlinks=true,
    linkcolor=teal,
    citecolor=magenta,
    }

\newtheorem{Theorem}{Theorem}

\newtheorem{conjecture}{Conjecture}
\newtheorem{proposition}{Proposition}[section]
\newtheorem{lemma}[proposition]{Lemma}

\newtheorem{theorem}[proposition]{Theorem}
\newtheorem{question}[conjecture]{Question}
\theoremstyle{definition}

\newcommand{\st}{\text{St}}

\newcommand{\Aut}{\text{Aut}}

\newcommand{\T}{\mathcal{T}}

\numberwithin{equation}{section}

\title[Generalized Brunner-Sidki-Vieira Groups]{The Hausdorff dimension of the generalized Brunner-Sidki-Vieira Groups}
\author{Jorge Fariña-Asategui and Mikel E. Garciarena}

\address{Jorge Fariña-Asategui: Centre for Mathematical Sciences, Lund University, 223 62 Lund, Sweden -- Department of Mathematics, University of the Basque Country UPV/EHU, 48080 Bilbao, Spain}
\email{jorge.farina\_asategui@math.lu.se}

\address{Mikel E. Garciarena: Department of Mathematics, University of Salerno, 84084 Fisciano (SA), Italy --  Department of Mathematics, University of the Basque Country UPV/EHU, 48080 Bilbao, Spain}
\email{mgarciarenaperez@unisa.it}

\keywords{Self-similar groups, weakly regular branch groups, Hausdorff dimension}
\subjclass[2020]{Primary: 20E08, 28A78; Secondary: 20E18}
\thanks{Both authors are supported by the Spanish Government, grant PID2020-117281GB-I00, partly with FEDER funds, and by the program ``Oberwolfach Research Fellows" by the Mathematisches Forschungsinstitut Oberwolfach in 2024. The first author also acknowledges support from the Walter Gyllenberg Foundation from the Royal Physiographic Society of Lund and the second author also acknowledges support from the “National Group for Algebraic and Geometric Structures, and their Applications” (GNSAGA - INdAM)}

\begin{document}
\maketitle

\begin{abstract}
	We compute the Hausdorff dimension of the closure of the generali-zed Brunner-Sidki-Vieira group acting on the $m$-adic tree for $m\ge 2$, providing the first examples of self-similar topologically finitely generated closed subgroups of transcendental Hausdorff dimension in the group of $m$-adic automorphisms.
\end{abstract}

\section{introduction}

Groups acting on regular rooted trees provide easy geometrical examples of groups with interesting properties, such as infinite finitely generated periodic groups \cite{GrigorchukBurnside} or groups of intermediate growth \cite{GrigorchukIntermediate}. The action of a group on a rooted tree can be described by looking at its action on every level of the tree. Therefore the automorphism group of a rooted tree is a profinite group with respect to the natural topology given by the \textit{level stabilizers} $\{\mathrm{St}(n)\}_{n\ge 1}$, where $\mathrm{St}(n)$ is the normal subgroup of automorphisms fixing every vertex at the $n$th level. The study of the Hausdorff dimension on profinite groups has generated a lot of interest in the recent years; see \cite{PadicAnalytic} for an overview and see \cite{IkerBenjamin, IkerAnitha1, IkerAnitha2, Jorge, OhianaAlejandraBenjamin, AndoniJon} for more recent work. It started with the work of Abercrombie \cite{Abercrombie} and with Barnea and Shalev \cite{BarneaShalev} proving that the Hausdorff dimension of a closed subgroup in a profinite group coincides with its lower box dimension. For $m\ge 2$, recall that the group of \textit{$m$-adic automorphisms} $\Gamma_m$ is defined as the group consisting of automorphisms whose local actions are given by powers of the cycle $(1\,\dotsb\, m)\in \mathrm{Sym}(m)$. For $G\le \Gamma_m$, the Hausdorff dimension of the closure of $G$ in $\Gamma_m$ is given by
$$\mathrm{hdim}_{\Gamma_m}(\overline{G})=\liminf_{n\to \infty}\frac{\log_m|G:\mathrm{St}_G(n)|}{\log_m|\Gamma_m:\mathrm{St}_{\Gamma_m}(n)|},$$
where $\st_G(n)$ and $\st_{\Gamma_m}(n)$ stand for $\st(n)\cap G$ and $\st(n)\cap\Gamma_m$ respectively.\\

Self-similar and branching groups are of particular interest as they appear naturally in the study of groups acting on regular rooted trees; see \cref{section: Preliminaries} for detailed definitions. Recent work of the first author \cite{Jorge} has shown that for both self-similar groups and branching groups the above lower limit is a proper limit and the Hausdorff dimension of the closure of such a group $G\le \Gamma_m$ is given by 
$$\mathrm{hdim}_{\Gamma_m}(\overline{G})=\log_m|G:\mathrm{St}_G(1)|-S_G(1/m),$$
where $S_G(x):=\sum_{n\ge 1}s_nx^n$ is the ordinary generating function of the sequence
$$s_n:=m\log_m|\mathrm{St}_G(n-1):\mathrm{St}_G(n)|-\log_m|\mathrm{St}_G(n):\mathrm{St}_G(n+1)|.$$

In this setting, any $\gamma\in [0,1]$ can be realized as the Hausdorff dimension of a topologically finitely generated group by the celebrated result of Abért and Virág; see \cite[Theorem~2]{AbertVirag} and see \cite{FGspectrum} for a constructive proof. The first author has proved that the same is true for closed self-similar groups \cite[Theorem A]{Jorge}. However, very few results are known on the Hausdorff dimension of both self-similar and topologically finitely generated closed groups. Indeed the only value which is known not to arise as the Hausdorff dimension of a self-similar topologically finitely generated closed group is~1; see \cite[Theorem 1.3]{Jorge}. Thus a first step is to understand the Hausdorff dimension of the closure of explicit examples of self-similar finitely generated groups.

For some families of self-similar groups the Hausdorff dimension of their closures has already been computed. Notable examples include the first and the second Grigorchuk group \cite{GrigorchukJustInfinite,Noce}, the Grigorchuk-Gupta-Sidki groups acting on the $p$-adic tree \cite{GGSHausdorff}, the Hanoi towers group \cite{Hanoi} and its generalizations \cite{Hanoi2} and the Basilica group \cite{BartholdiHausdorff} and its different generalizations \cite{pBasilica,GeneralizedBasilica}. 

In this paper we study the Brunner-Sidki-Vieira group $H$ introduced over the binary rooted tree in \cite{BSV1} and generalized to the $m$-adic tree in \cite{BSV2}. The group $H$ is generated by the automorphisms $a$ and $b$ defined recursively as $\psi(a)=(1,\dotsc,1,a)\sigma$ and $\psi(b)=(1,\dotsc,1,b^{-1})\sigma$ for $\sigma:=(1\,2\,\dotsb\, m)\in \mathrm{Sym}(m)$; see \cref{section: Preliminaries} for further details. The group $H$ satisfies many structural properties, such as being self-similar, super strongly fractal, level-transitive and saturated, as shown in \cref{proposition: self-similar etc m}. Its automorphism group is also just equal to its normalizer in the full automorphism group; see \cref{proposition: automorphisms of H}.

The Hausdorff dimension of the closure in $\Gamma_2$ of the Brunner-Sidki-Vieira group acting on the binary rooted tree  is known to be 1/3; see \cite[Example~2.4.5]{BartholdiHausdorff}. For $m\ge 3$, it was never computed. We close this gap here.

 \begin{Theorem}
 \label{Theorem: Hausdorff dimension and congruence orders for m}
	Let $H$ be the generalized Brunner-Sidki-Vieira group acting on the $m$-adic tree. Then the Hausdorff dimension of its closure in $\Gamma_m$ is
	\begin{align*}
        \mathrm{hdim}_{\Gamma_m}(\overline{H})=
        \frac{m-\tau(m-1)\log_m 2}{m+1},
\end{align*}
where the parameter $\tau$ is defined as
\begin{align}
\label{align: definition of tau}
    \tau:=\begin{cases}
        0& \text{if }m \text{ is odd},\\
        1& \text{if }m \text{ is even}.
    \end{cases}
\end{align} 
\end{Theorem}

For $m$ odd the Hausdorff dimension of the closure of the generalized Brunner-Sidki-Vieira group coincides with the Hausdorff dimension of the closure of the second Basilica group of the odometer; for $m=p$ a prime see \cite[Theorem C]{pBasilica} and for any odd~$m$ see \cite[Theorem~1.7]{GeneralizedBasilica}. However, for $m$ even and not a power of~2 this value is transcendental by the Gelfond-Schneider Theorem; see the discussion at the end of \cref{section: Hausdorff dimension}. Therefore we obtain the first examples in the literature of self-similar topologically finitely generated closed subgroups of $\Gamma_m$ with transcendental Hausdorff dimension. Note that irrationality of the Hausdorff dimension comes from the logarithms themselves, as in the case of the Hanoi towers group \cite{Hanoi}. This is fundamentally different from the self-similar not topologically finitely generated closed groups constructed in \cite{Jorge}, where the irrationality of the Hausdorff dimension is due to the sequence $\{s_n\}_{n\ge 1}$ not becoming eventually periodic. Even if the Hausdorff dimension of self-similar topologically finitely generated closed subgroups may not be rational all the studied examples have eventually periodic sequences $\{s_n\}_{n\ge 1}$ and, therefore, their generating functions $S_G(x)$ are indeed rational functions. Hence we ask whether this is a general trait of self-similar topologically finitely generated closed groups.

\begin{question}
\label{question: rational generating function}
For a self-similar topologically finitely generated closed subgroup $G$ is $S_G(x)$ always a rational function?
\end{question}

\subsection*{\textit{\textmd{Organization}}} 
In \cref{section: Preliminaries} we give the basic definitions and results needed in the subsequent sections. We also define the generalized Brunner-Sidki-Vieira group and give its first structural properties. In \cref{section: Hausdorff dimension} we compute the Hausdorff dimension of the closure of this group, proving \cref{Theorem: Hausdorff dimension and congruence orders for m}. 

\subsection*{\textit{\textmd{Notation}}} 
The integer $m\ge 2$ will be used to denote the degree of an arbitrary regular rooted tree. For a normal subgroup $N\le G$ we shall use the notation $g\equiv_N h$ for $gh^{-1}\in N$. For two elements $f,g$ of some group $G$ we use the notation $f^{-g}:=(f^g)^{-1}$.

\subsection*{Acknowledgements} We would like to thank Marialaura Noce and Anitha Thillai-sundaram for suggesting the topic and for very helpful discussions. We would also like to thank Gustavo A. Fernández-Alcober for providing valuable feedback on the writing of the present paper.

\section{Preliminaries and the generalized Brunner-Sidki-Vieira group}
\label{section: Preliminaries}

\subsection{The full automorphism group} For $m\ge 2$ we define the $m$-\textit{adic tree} $\mathcal{T}$ as the infinite regular rooted tree where each vertex has exactly $m$ descendants. Each vertex can be identified with a finite word in letters from the set $X=\{1,\dotsc,m\}$. Two vertices $u$ and $v$ are joined by an edge if $u=vx$ with $x\in X$. In this case, we say that $u$ is \textit{below} $v$. The \textit{$n$th level} of the tree is defined to be the set of all vertices that consist of words of length $n$, and the $n$th \textit{truncated tree} $\mathcal{T}_n$ will be the tree consisting of words of length at most $n$. We also use the word level to refer to the number~$n$. Let $\mathrm{Aut}~\mathcal{T}$ be the group of graph automorphisms of the $m$-adic tree, i.e. those bijective maps on the set of vertices of the tree preserving adjacency. Since we are considering regular rooted trees such automorphisms fix the root and permute the vertices at the same level of the tree. 

Let $g\in{\mathrm{Aut}~\mathcal{T}}$. For any vertex $v$ and any integer $1\le k\le \infty$, we define~the~\emph{section of $g$ at $v$ of depth }$k$ as the unique automorphism $g|_v^k$, of the truncated tree $\T_k$ (we fix $\mathcal{T}_\infty=\mathcal{T}$ for $k=\infty$), such that $(vu)g=(v)g\cdot(u)g|_v^k$ for every $u$ of length at most $k$, for finite $k$, and of arbitrary length for $k=\infty$. For $k=\infty$ we write  $g|_v$ for simplicity and we will refer to it as \textit{section of $g$ at $v$}  and for $k=1$ we call it the \textit{label of $g$ at $v$}.

For every $n\ge 1$ the subset $\mathrm{St}(n)$ of automorphisms fixing the $n$th level of~the tree pointwise forms a normal subgroup of $\Aut~\T$ of finite index called the \textit{$n$th level stabilizer}. Similarly, for any vertex $v$ of the tree we can define its \textit{vertex stabilizer} $\mathrm{st}(v)$ as the subgroup of automorphisms fixing the vertex $v$. For each vertex $v$, we shall define the continuous homomorphism $\psi_v:\mathrm{st}(v)\to {\mathrm{Aut}~\mathcal{T}}$ via $g\mapsto g|_v$ and for each integer $n\ge 1$ we define the continuous isomorphisms $\psi_n:\mathrm{St}(n)\to {\mathrm{Aut}~\mathcal{T}}\times\dotsb\times {\mathrm{Aut}~\mathcal{T}}$ via $g\mapsto (g|_{v_1},\dotsc,g|_{v_{m^n}})$, where $v_1,\dotsc,v_{m^n}$ denote the vertices at the $n$th level of the tree from left to right. We shall write $\psi:=\psi_1$ for simplicity. Note that $\psi$ extends to an isomorphism $\psi:\mathrm{Aut}~\mathcal{T}\to  {\mathrm{Aut}~\mathcal{T}}\wr \mathrm{Sym}(m)$ where $\mathrm{Sym}(m)$ is the symmetric group on $m$ symbols. We shall use $\psi$ to define automorphisms of the $m$-adic tree recursively.

\subsection{Subgroups of $\mathrm{Aut}~\mathcal{T}$}
Let $G$ be a subgroup of ${\mathrm{Aut}~\mathcal{T}}$. We define the subgroups $\mathrm{st}_G(v):=\mathrm{st}(v)\cap G$ and $\mathrm{St}_G(n):=\mathrm{St}(n)\cap G$ for any vertex $v$ and level $n\ge 1$ respectively. We say that $G$ is \textit{self-similar} if for any $g\in G$ we have $g|_{v}\in G$ for any vertex $v$ of the tree. We shall say that $G$ is \textit{fractal} (or \textit{strongly fractal}) if $G$ is self-similar and $\psi_v(\mathrm{st}_G(v))=G$ (respectively $\psi_v(\mathrm{St}_G(1))=G$) for any $v$ at the first level of the tree. An even stronger version of fractality is \textit{super strongly fractal}, where $\psi_v(\mathrm{St}_G(n))=G$ for every vertex $v$ at the $n$th level of the tree for every level $n\ge 1$; see \cite{JoneFractal} for a detailed discussion on the different concepts of fractality.

We say that $G$ is \textit{level-transitive} if $G$ acts transitively on every level of the tree. For any vertex $v$ let $\mathrm{rist}_G(v)$ be the associated \textit{ rigid vertex stabilizer}, i.e. the subgroup of $\mathrm{st}_G(v)$ consisting of automorphisms $g\in \mathrm{st}_G(v)$ only acting non-trivially on the vertices in the subtree rooted at $v$. Then for any $n\ge 1$ we define the corresponding \textit{rigid level stabilizers} $\mathrm{Rist}_G(n)$ as the product of all rigid vertex stabilizers of the vertices at level~$n$. We say that $G$ is \textit{weakly branch} if $G$ is level-transitive and its rigid level stabilizers are non-trivial for every level, and \textit{branch} if they are furthermore of finite index in $G$. \textit{Weakly regular branch} and \textit{regular branch} are defined as $G$ being a self-similar, level-transitive group containing a \textit{branching} subgroup $K$, i.e. a subgroup such that $\psi(\mathrm{St}_K(1))\ge K\times\dotsb\times K$, where $K$ is assumed to be non-trivial and of finite index respectively. In general, a group $K$ satisfying the assumption $\psi(\mathrm{St}_K(1))\ge K\times\dotsb\times K$ is called a \textit{branching} group.

Finally, a group $G\le \mathrm{Aut}~\mathcal{T}$ is said to be \textit{saturated} if for any positive integer $n$ there exists a characteristic subgroup $G_n\le\mathrm{St}_G(n)$ such that $G_n$ is level-transitive on every subtree rooted at the $n$th level.

\subsection{Hausdorff dimension}
Countably based profinite groups can be endowed with a metric which allows the definition of Hausdorff measures and therefore a Hausdorff dimension for its subgroups. It was proved by Barnea and Shalev in~\cite{BarneaShalev} based on the work of Abercrombie \cite{Abercrombie} that for a profinite group the Hausdorff dimension of its closed subgroups coincides with their lower box dimension. The group ${\mathrm{Aut}~\mathcal{T}}$ is residually finite as $\bigcap \mathrm{St}(n)=1$. What is more ${\mathrm{Aut}~\mathcal{T}}$ is a countably based profinite group with respect to the topology induced by this filtration of the level stabilizers. For a group $G$ such that $g|_v^1\in H\le \mathrm{Sym}(m)$ for every element $g\in G$ and every vertex~$v$ of the tree, the Hausdorff dimension of $\overline{G}$ may be computed relative to the \textit{iterated permutational wreath product} $W_H:=\varprojlim_{n\ge 1} H\wr \overset{n}{\dotsb}\wr H$ as
$$\mathrm{hdim}_{W_H}(\overline{G})=\liminf_{n\to \infty}\frac{\log_m|G:\mathrm{St}_G(n)|}{\log_m|W_H:\mathrm{St}_{W_H}(n)|}.$$
The following theorem was proven by the first author in \cite{Jorge}.

\begin{theorem}[{see \cite[Theorem B]{Jorge}}]
If $G\le \Gamma_m$ is either self-similar or branching then the Hausdorff dimension of the closure of $G$ in $\Gamma_m$ is
$$\mathrm{hdim}_{\Gamma_m}(\overline{G})=\log_m|G:\mathrm{St}_G(1)|-S_G(1/m),$$
where $S_G(x):=\sum_{n\ge 1}s_nx^n$ is the ordinary generating function of the sequence
$$s_n:=m\log_m|\mathrm{St}_G(n-1):\mathrm{St}_G(n)|-\log_m|\mathrm{St}_G(n):\mathrm{St}_G(n+1)|.$$
\end{theorem}

For a branching group $G$ an application of the Second Isomorphism Theorem yields
\begin{align}
\label{align: sn for branching}
    s_n=-\log_m|\psi^{-1}(G\times\dotsb\times G)\mathrm{St}_G(n):\psi^{-1}(G\times\dotsb\times G)\mathrm{St}_G(n+1)|;
\end{align}
see \cite[Section 3]{Jorge} for further details.

\subsection{The generalized Brunner-Sidki-Vieira group $H$}
The generalized Brunner-Sidki-Vieira group $H$ acting on the $m$-adic tree is generated by the automorphisms~$a$ and $b$ defined recursively as $\psi(a)=(1,\dotsc,1,a)\sigma$ and $\psi(b)=(1,\dotsc,1,b^{-1})\sigma$, where $\sigma=(1\,\dotsb \,m)\in \mathrm{Sym}(m)$; see \cref{fig: BSV generalized}. Clearly $H\le \Gamma_m$. We fix the element $\lambda:=ab^{-1}$. Note that $\lambda\in \st_H(1)$ and that $\psi(\lambda)=(1,\dots, 1,ab)$.

The generalized Brunner-Sidki-Vieira group $H$ satisfies the following properties.

\begin{proposition}
\label{proposition: self-similar etc m}
	The group $H$ is self-similar, super strongly fractal, level-transitive and saturated.
\end{proposition}
\begin{proof}
	The sections of the generators of $H$ are again in $H$; thus, any section of any element of $H$ is again in $H$ proving $H$ is self-similar. Note that $H=\langle a,b^{-1}\rangle$, and since for every $k\in \mathbb{N}$ we have both $\psi_k(a^{m^k})=(a,\dotsc,a)$ and $\psi_k(b^{m^k})=(b^{-1},\dotsc,b^{-1})$ with $a^{m^k},b^{m^k}\in\mathrm{St}_{H}(k)$, the group $H$ is super strongly fractal. Since $H$ acts transitively on the first level of the tree and $H$ is fractal it is level-transitive by \cite[Lemma 2.7]{JoneFractal}. Since the cyclic subgroup generated by $a$ is level-transitive, and the subgroup $H^{m^k}\le \mathrm{St}_H(k)$ contains $\langle a^{m^k}\rangle$, it acts transitively on all the subtrees at the $n$th level of the tree. And because $H^{m^k}$ is characteristic in $H$ this implies that $H$ is saturated.
\end{proof}

\begin{figure}[H]
\begin{center}
		\begin{forest}
				for tree={circle, fill=teal, inner sep=1.2pt, outer sep=0pt, s sep=5mm,}
				[,name=root,
				[, name=a1]
				[, name=a2]
				[, name=a3]
				[, name=a4, 
				[, name=b1]
				[, name=b2]
				[, name=b3]
				[,name=b4,
				[, name=c1]
				[, name=c2]
				[, name=c3]
				[,name=c4, 
				[,name=d1][,name=d2][,name=d3][,name=d4 [, fill=white, edge=white][,fill=white, edge=white][,fill=white, edge=white][,fill=white, edge=dotted]]]]]]
				]
			    \draw[-latex, color=Rhodamine] (a1.45) to[out=45,in=135,looseness=1] (a2.135);
			    \draw[-latex, color=Rhodamine] (a2.45) to[out=45,in=135,looseness=1] (a3.135);
			    \draw[-latex, color=Rhodamine] (a3.45) to[out=45,in=135,looseness=1] (a4.135);
			    \draw[-latex, color=Rhodamine] (a4.-150) to[out=-150,in=-30,looseness=1] (a1.-30);
			    \draw[-latex, color=Rhodamine] (b1.45) to[out=45,in=135,looseness=1] (b2.135);
			    \draw[-latex, color=Rhodamine] (b2.45) to[out=45,in=135,looseness=1] (b3.135);
			    \draw[-latex, color=Rhodamine] (b3.45) to[out=45,in=135,looseness=1] (b4.135);
			    \draw[-latex, color=Rhodamine] (b4.-150) to[out=-150,in=-30,looseness=1] (b1.-30);
			    \draw[-latex, color=Rhodamine] (c1.45) to[out=45,in=135,looseness=1] (c2.135);
			    \draw[-latex, color=Rhodamine] (c2.45) to[out=45,in=135,looseness=1] (c3.135);
			    \draw[-latex, color=Rhodamine] (c3.45) to[out=45,in=135,looseness=1] (c4.135);
			    \draw[-latex, color=Rhodamine] (c4.-150) to[out=-150,in=-30,looseness=1] (c1.-30);
			    \draw[-latex, color=Rhodamine] (d1.45) to[out=45,in=135,looseness=1] (d2.135);
			    \draw[-latex, color=Rhodamine] (d2.45) to[out=45,in=135,looseness=1] (d3.135);
			    \draw[-latex, color=Rhodamine] (d3.45) to[out=45,in=135,looseness=1] (d4.135);
			    \draw[-latex, color=Rhodamine] (d4.-150) to[out=-150,in=-30,looseness=1] (d1.-30);
			    \node (a) [above=0cm of root]{$a$};
			\end{forest}\quad
		\begin{forest}
				for tree={circle, fill=teal, inner sep=1.5pt, outer sep=0pt, s sep=5mm,}
				[,name=root,
				[, name=a1]
				[, name=a2]
				[, name=a3]
				[, name=a4, 
                [, name=b1 , 
				[, name=c1]
				[, name=c2]
				[, name=c3]
				[,name=c4, 
				[,name=d1 , [, fill=white, edge=white][,fill=white, edge=white][,fill=white, edge=white][,fill=white, edge=dotted]][,name=d2][,name=d3][,name=d4]]]
				[, name=b2]
                [, name=b3 ]
				[,name=b4]
				]
				]
			    \draw[-latex, color=Rhodamine] (a1.45) to[out=45,in=135,looseness=1] (a2.135);
			    \draw[-latex, color=Rhodamine] (a2.45) to[out=45,in=135,looseness=1] (a3.135);
			    \draw[-latex, color=Rhodamine] (a3.45) to[out=45,in=135,looseness=1] (a4.135);
			    \draw[-latex, color=Rhodamine] (a4.-150) to[out=-150,in=-30,looseness=1] (a1.-30);
			    \draw[-latex, color=Melon] (b4.45) to[out=135,in=45,looseness=1] (b3.45);
			    \draw[-latex, color=Melon] (b3.135) to[out=135,in=45,looseness=1] (b2.45);
			    \draw[-latex, color=Melon] (b2.135) to[out=135,in=45,looseness=1] (b1.45);
			    \draw[-latex, color=Melon] (b1.-30) to[out=-30,in=-150,looseness=1] (b4.-150);
			    \draw[-latex, color=Rhodamine] (c1.45) to[out=45,in=135,looseness=1] (c2.135);
			    \draw[-latex, color=Rhodamine] (c2.45) to[out=45,in=135,looseness=1] (c3.135);
			    \draw[-latex, color=Rhodamine] (c3.45) to[out=45,in=135,looseness=1] (c4.135);
			    \draw[-latex, color=Rhodamine] (c4.-150) to[out=-150,in=-30,looseness=1] (c1.-30);
			    \draw[-latex, color=Melon] (d4.45) to[out=135,in=45,looseness=1] (d3.45);
			    \draw[-latex, color=Melon] (d3.135) to[out=135,in=45,looseness=1] (d2.45);
			    \draw[-latex, color=Melon] (d2.135) to[out=135,in=45,looseness=1] (d1.45);
			    \draw[-latex, color=Melon] (d1.-30) to[out=-30,in=-150,looseness=1] (d4.-150);
			    \node (b) [above=0cm of root]{$b$};
			\end{forest}
			\end{center}
    \caption{The generators of the generalized Brunner-Sidki-Vieira group acting on the 4-adic tree.}
    \label{fig: BSV generalized}
\end{figure}
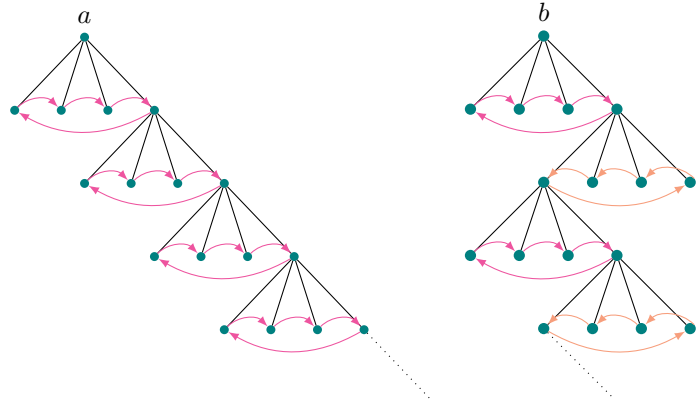

Since $H$ is weakly branch, as shown below in \cref{proposition: branch structure BSV m}, the following result follows from \cite[Theorem 7.5]{Nekrashevych}, generalizing \cite[Proposition 8.7]{Nekrashevych} to any $m\ge 2$.

\begin{proposition}
\label{proposition: automorphisms of H}
    The automorphism group of $H$ is given by $\mathrm{Aut}(H)=N_{\mathrm{Aut}~\mathcal{T}}(H)$.
\end{proposition}

We conclude this section by presenting to the reader some structural properties of $H$ proved by Sidki and Silva in  \cite{BSV2}, which are needed to compute the Hausdorff dimension of the closure of $H$.  We define~$L_2$ via
$$L_2:=\psi^{-1}\big(H'\times\stackrel{m}{\cdots}\times H'\big).$$
Note that in \cite[Theorem A]{BSV2} it is shown that $H$ is weakly regular branch over $H'$, therefore $L_2$ is contained in $H'$. We refer the reader to \cite{BSV2} for the proofs of the following propositions.

\begin{proposition}[see {\cite[Proof of Theorem 1]{BSV2}}]
    \label{proposition: H/H' free abelian of rank 2} The quotient $H/H'=\langle aH'\rangle \times\langle bH'\rangle$ is free abelian of rank 2.
\end{proposition}

\begin{proposition}[see {\cite[Theorem A]{BSV2}}]
	\label{proposition: branch structure BSV m}
	The commutator subgroup $H'$ factors as 
 $$H'=\langle [\lambda,a],[\lambda,a]^a,\dotsc,[\lambda,a]^{a^{m-2}}\rangle \ltimes L_2,$$
	and $\langle [\lambda,a],[\lambda,a]^a,\dotsc,[\lambda,a]^{a^{m-2}}\rangle$ is free abelian of rank $m-1$. In particular, the group $H$ is weakly regular branch over $H'$. Furthermore
 $$\psi(H')\subseteq\{(h_1,\dots, h_m):h_i\in H,~ h_1\cdots h_m\in H'\}.$$
\end{proposition}

We note that the last statement in \cref{proposition: branch structure BSV m} is a simple observation which was not stated in \cite[Theorem A]{BSV2}.

\section{Proof of the main result}
\label{section: Hausdorff dimension}

We prove first that for a finitely generated group the Hausdorff dimension of its closure in $\Gamma_m$ coincides with the one of its commutator subgroup. This was proved for closed subgroups of $\Gamma_p$ by Abért and Virág in \cite[Theorem 5]{AbertVirag}.

\begin{proposition}
\label{proposition: G and G' same hausdorff dimension}
Let $G\le \Gamma_m$ be finitely generated. Then the Hausdorff dimension of $\overline{G}$ equals the one of $\overline{G'}$. In particular if $\overline{G}$ is solvable, then $\overline{G}$ has trivial Hausdorff dimension.
\end{proposition}
\begin{proof}
Since
$$\log_m|G:\mathrm{St}_G(n)|=\log_m|G:G'\mathrm{St}_G(n)|+\log_m|G':\mathrm{St}_{G'}(n)|,$$
it is enough to prove that the first logarithmic index in the above sum is negligible for sufficiently large values of $n$. To this end, suppose $G$ is generated by $d$ elements. Then we get
\begin{align*}
    \log_m|G:G'\mathrm{St}_G(n)|\le \log_m m^{dn}=dn,
\end{align*}
as each generator of the quotient can have at most order $m^n$ since they are acting faithfully on a set of $m^n$ elements. However, the growth of $\log_m|\Gamma_m:\mathrm{St}_{\Gamma_m}(n)|$ is exponential in $n$ so the result follows.
\end{proof}

Therefore by \cref{proposition: G and G' same hausdorff dimension} the Hausdorff dimension of the closure of $H$ coincides with the one of the closure of $H'$. The latter is a branching group, so to compute the Hausdorff dimension of its closure it is enough to compute the orders of the quotients $H'/L_2\mathrm{St}_{H'}(n)$ by \cref{align: sn for branching}. As will become apparent later, we need to compute the orders of $ab$ and $\lambda$ modulo $H'\mathrm{St}_H(n)$.

\begin{lemma}
\label{Lemma: quotient direct product of finite cyclic groups}
The quotient $H'/L_2 \mathrm{St}_{H'}(n)$ is a direct product of $m-1$ cyclic groups of order $\ell$, where $\ell$ is the order of $ab$ modulo $H'\mathrm{St}_H(n-1)$.
\end{lemma}
\begin{proof}
The monomorphism $\psi:\st_H(1)\rightarrow H\times\stackrel{m}{\cdots}\times H$ induces the isomorphism, which we also denote $\psi$ by a slight abuse of notation,
\begin{align}\label{al: isomorphism}
    \psi:\frac{H'}{L_2\st_{H'}(n)}\overset{\cong}{\longrightarrow} \frac{\psi(H')}{\psi(L_2\st_{H'}(n))}.
\end{align}
Now note that 
\begin{align*}
    \psi(L_2\st_{H'}(n))&=\left(H'\times\stackrel{m}{\cdots} \times H'\right)\psi(\st_{H'}(n))\\
    &=\left(H'\times\stackrel{m}{\cdots}\times H'\right)\left(\left(\st_{H}(n-1)\times\stackrel{m}{\cdots}\times \st_{H}(n-1)\right)\cap \psi(H')\right)\\
    &=\left(H'\st_H(n-1)\times\stackrel{m}{\cdots}\times H'\st_H(n-1)\right)\cap \psi(H'),
\end{align*}
where the last equality follows from Dedekind's Modular Law as $H'\ge L_2$ by \cref{proposition: branch structure BSV m}. Therefore the isomorphism in \textcolor{teal}{(}\ref{al: isomorphism}\textcolor{teal}{)} induced by $\psi$ yields
\begin{align*}
   \psi:\frac{H'}{L_2\st_{H'}(n)}&\overset{\cong}{\longrightarrow} \frac{\psi(H')}{\left(H'\st_H(n-1)\times\stackrel{m}{\cdots}\times H'\st_H(n-1)\right)\cap \psi(H')}\\
   &\overset{\cong}{\longrightarrow} \frac{\psi(H')\left(H'\st_H(n-1)\times\stackrel{m}{\cdots}\times H'\st_H(n-1)\right)}{\left(H'\st_H(n-1)\times\stackrel{m}{\cdots}\times H'\st_H(n-1)\right)}
\end{align*}
by the second isomorphism theorem. Thus $\psi$ induces an isomorphism between $H'/L_2\st_{H'}(n)$ and its image in
\begin{align*}
\frac{H\times\stackrel{m}{\cdots}\times H}{H'\st_H(n-1)\times\stackrel{m}{\cdots}\times H'\st_H(n-1)},
\end{align*}
where the image via $\psi$ is generated by the images via $\psi$ of the elements $[\lambda,a]^{a^i}$ with $0\leq i\leq m-2$, in other words
\begin{align*}
\psi([\lambda,a]^{a^i})&\equiv_{\psi(L_2)}(1,\stackrel{i-1}{\ldots},1,(ba)^{-1},ba,1,\dots, 1).
\end{align*}
If the order of $ba$ modulo $H'\st_H(n-1)$ is $\ell$, we obtain an isomorphism between $H'/L_2\st_{H'}(n)$ and the subgroup of $\mathbb{Z}/\ell\mathbb{Z}\times\stackrel{m}{\cdots}\times\mathbb{Z}/\ell\mathbb{Z}$ corresponding to the image of the matrix 
\begin{align*}
        A:=\begin{pmatrix}
        1&0&0&\dotsb &0&0&-1\\
        -1&1&0&\dotsb &0&0&0\\
        0&-1&1&\dotsb &0&0&0\\
        \vdots &\vdots&\vdots&\ddots&\vdots& \vdots&\vdots\\
        0&0&0&\dotsb & 1&0&0\\
        0&0&0&\dotsb & -1&1&0
    \end{pmatrix}.
    \end{align*}
    By Gaussian elimination we get 
    \begin{align}\label{matix A}
       A\sim \begin{pmatrix}
        1&0&0&\dotsb &0&0&-1\\
        0&1&0&\dotsb &0&0&-1\\
        0&0&1&\dotsb &0&0&-1\\
        \vdots &\vdots&\vdots&\ddots&\vdots& \vdots&\vdots\\
        0&0&0&\dotsb & 1&0&-1\\
        0&0&0&\dotsb & 0&1&-1
    \end{pmatrix}
    \end{align}
showing that 
\begin{align*}
    H'/L_2\st_{H'}(n) &\cong \mathbb{Z}/\ell\mathbb{Z}\times\stackrel{m-1}{\cdots}\times\mathbb{Z}/\ell\mathbb{Z}.\qedhere
\end{align*}
\end{proof}

\begin{proposition}
\label{proposition: orders of elements in abelianization m}
We have the following:
\begin{enumerate}[\normalfont(i)]
    \item if $m\geq 3$ is odd, for every $k\ge 0$ and $1\le j\le m-1$ we get
    \begin{enumerate}[(P1)]
    \vspace{5pt}
        \item\label{P1} $\quad(ab)^{jm^{k}}\in H'\mathrm{St}_H(2k)\setminus H'\mathrm{St}_H(2k+1)$
        \vspace{5pt}
\item\label{P2} $\quad\lambda^{jm^{k}}\in H'\mathrm{St}_H(2k+1)\setminus H'\mathrm{St}_H(2k+2);$
    \end{enumerate}
    \vspace{5pt}
\item if $m\geq 2$ is even, for every $k\ge 0$ we get
\vspace{5pt}
\begin{enumerate}[(Q1)]
    \item $\quad(ab)^{m^{k+1}2^{-1}}\in H'\mathrm{St}_H(2k+1)\setminus H'\mathrm{St}_H(2k+2)$
   \vspace{5pt}
\item $\quad\lambda^{m^{k+1}2^{-1}}\in H'\mathrm{St}_H(2k+2)\setminus H'\mathrm{St}_H(2k+3).$
\end{enumerate}
   \vspace{5pt}
Furthermore, if $m\ge 4$ is even, then for every $k\ge 0$ and $1\leq j\leq m/2-1$ or $m/2+1\leq j\leq m-1$ we get
   \vspace{5pt}
\begin{enumerate}[(R1)]
    \item  $\quad(ab)^{jm^{k}}\in H'\mathrm{St}_H(2k)\setminus H'\mathrm{St}_H(2k+1)$
       \vspace{5pt}
\item$\quad\lambda^{jm^{k}}\in H'\mathrm{St}_H(2k+1)\setminus H'\mathrm{St}_H(2k+2).$
\end{enumerate}
\end{enumerate}
\end{proposition}
\begin{proof}
We will prove the proposition by induction over $k$, but first let us sketch the inductive structure of the proof for the convenience of the reader. On the one hand, we establish the base cases for both \textcolor{teal}{\textit{(P1)}} and \textcolor{teal}{\textit{(P2)}}. Then, we prove \textcolor{teal}{\textit{(P1)}} and \textcolor{teal}{\textit{(P2)}} simultaneously. On the other hand, we establish the base cases for \textcolor{teal}{\textit{(R1)}}, \textcolor{teal}{\textit{(R2)}}, \textcolor{teal}{\textit{(Q1)}} and \textcolor{teal}{\textit{(Q2)}}, in that specified order. Finally we proceed with the proof of \textcolor{teal}{\textit{(R1)}} and \textcolor{teal}{\textit{(Q1)}}, and use both of them to prove \textcolor{teal}{\textit{(R2)}} and \textcolor{teal}{\textit{(Q2)}}. Note that the methods used to prove \textcolor{teal}{\textit{(P1)}}, \textcolor{teal}{\textit{(R1)}}, and \textcolor{teal}{\textit{(Q1)}} are closely aligned. The same happens with \textcolor{teal}{\textit{(P2)}}, \textcolor{teal}{\textit{(R2)}}, and \textcolor{teal}{\textit{(Q2)}}.

Now let us proceed with the actual proof. First of all, observe that for any $m\ge 2$ we have
\begin{align}
\label{align: ab order mod H'St(k)}
    \psi((ab)^{m})\equiv_{\psi(H')}(\lambda,\dotsc,\lambda)\quad\text{and}\quad
    \psi(\lambda)=(1,\dotsc,1,ab),
\end{align}
where the first equality follows from
\begin{align*}
    \psi(a^m)&=(a,\dots, a)\quad\text{and}\quad \psi(b^m)=(b^{-1},\dots, b^{-1}).
\end{align*}
The base case, i.e. $k=0$, for \textcolor{teal}{\textit{(P1)}} follows from the quotient $H/\mathrm{St}_H(1)$ being cyclic of order $m\geq 3$ generated by $a\mathrm{St}_H(1)=b\mathrm{St}_H(1)$ and \cref{align: ab order mod H'St(k)}. Indeed since $m$ is odd then $m$ and 2 are coprime and $ab$ has order $m$ modulo $H'\mathrm{St}_H(1)$. For \textcolor{teal}{\textit{(P2)}}, it is clear that $\lambda^j\in H'\st_H(1)$. Now assume by contradiction that $\lambda^j\in \st_H(2)H'$, then from \cref{Lemma: quotient direct product of finite cyclic groups} we deduce that $\psi(\lambda^j)=(h_1,\dots, h_m)$ with $h_1\cdots h_m\in H'\st_H(1)$, but by \cref{align: ab order mod H'St(k)} we know that $\psi(\lambda^j)=(1,\dots,1, (ab)^j)$ so $(ab)^j\in H'\st_H(1)$ which implies that $j\geq m$ arriving to a contradiction. Then  by induction on $k$, for $1\le j\le m-1$ we get
\begin{align}
\label{align: induction ab}
    (ab)^{jm^{k-1}}\in H'\mathrm{St}_H(2k-2)\setminus H'\mathrm{St}_H(2k-1)
\end{align}
and 
\begin{align}
\label{align: induction lambda}
    \lambda^{jm^{k-1}}\in H'\mathrm{St}_H(2k-1)\setminus H'\mathrm{St}_H(2k).
\end{align}

We prove first \textcolor{teal}{\textit{(P1)}}. By \textcolor{teal}{Equations (}\ref{align: ab order mod H'St(k)}\textcolor{teal}{)} and \textcolor{teal}{(}\ref{align: induction ab}\textcolor{teal}{)}, there exists some $h\in H'$ such that
\begin{align*}
    \psi_2(h(ab)^{jm^{k}})&\in \big(H'\mathrm{St}_H(2k-2)\times\stackrel{m^2}{\cdots}\times H'\mathrm{St}_H(2k-2)\big)\cap \psi_2(\mathrm{St}_H(2)).
\end{align*}
Therefore, using that $H$ is weakly regular branch over $H'$ by \cref{proposition: branch structure BSV m},  we get 
$$(ab)^{jm^{k}}\in H'\mathrm{St}_H(2k)$$
from the injectivity of the map $\psi_2$. Hence we just need to prove that
$$(ab)^{jm^{k}}\notin H'\mathrm{St}_H(2k+1)$$
for any $1\le j\le m-1$. Assume by contradiction that this is not the case. Thus using \cref{align: ab order mod H'St(k)} and \cref{proposition: branch structure BSV m}, there exists some $c\in \langle [\lambda,a]^{a^r}\mid 0\le r\le m-2\rangle$ such that 
\begin{align*}
    (\lambda^{jm^{k-1}},\dots, \lambda^{jm^{k-1}})\psi(c)\in H'\st_H(2k)\times \stackrel{m}{\cdots}\times H'\st_H(2k)
\end{align*}
where $\psi(c)\equiv_{\psi(L_2)}\big((ab)^{\ell_1},\dots, (ab)^{\ell_m}\big)$ with $\ell_i\in \mathbb{Z}$ as
\begin{align*}
    \psi(H')&\le \langle ab\rangle H'\times\stackrel{m}{\cdots}\times \langle ab\rangle H'.
\end{align*}
Therefore we deduce that there exists $\ell\in \mathbb{Z}$ such that
$$(ab)^\ell\lambda^{jm^{k-1}}\in H'\mathrm{St}_H(2k).$$
By \textcolor{teal}{(}\ref{align: induction lambda}\textcolor{teal}{)} we must have
$$(ab)^\ell\in H'\mathrm{St}_H(2k-1)\setminus H'\mathrm{St}_H(2k).$$
However, this is impossible. Indeed, suppose that $m^i$ is the largest power of $m$ dividing $\ell$. If $i\geq k$ then $(ab)^\ell\in H'\mathrm{St}_H(2k)$ leading to a contradition. If $i<k$ then by the induction hypothesis $(ab)^\ell\in H'\st_{H}(2i)\setminus H'\st_H(2i+1)$
and $$(ab)^\ell\not\in H'\st_H(2i+1)\geq H'\st_H(2k-1)$$ leading to a contradiction.

For \textcolor{teal}{\textit{(P2)}}, the containment $\lambda^{jm^{k}}\in H'\mathrm{St}_H(2k+1)$ follows from \textcolor{teal}{\textit{(P1)}} and \cref{align: ab order mod H'St(k)}. To prove that $\lambda^{jm^k}\notin H'\mathrm{St}_H(2k+2)$ we assume by contradiction that this is not the case. Therefore, as before, using \cref{align: ab order mod H'St(k)} and \cref{proposition: branch structure BSV m}, there exists some $c\in \langle [\lambda,a]^{a^r}\mid 0\le r\le m-2\rangle$ such that 
\begin{align*}
    (1,\dots,1, (ab)^{jm^{k}})\psi(c)\in H'\st_{H}(2k+1)\times \stackrel{m}{\cdots}\times H'\st_H(2k+1)
\end{align*}
where $\psi(c)\equiv_{\psi(L_2)}\big((ab)^{\ell_1},\dots, (ab)^{\ell_m}\big)$ with $\ell_i\in \mathbb{Z}$. From \textcolor{teal}{\textit{(P1)}}, we deduce that in order to have 
\begin{align*}
  (ab)^{jm^{k}}  (ab)^{\ell_m}=(ab)^{jm^{k}+\ell_m}\in H'\st_H(2k+1),
\end{align*}
the integer $jm^{k}+\ell_m$ has to be divisible by $m^{k+1}$. In other words $jm^{k}+\ell_m\equiv 0 \mod m^{k+1}$. Similarly  for 
\begin{align*}
  (ab)^{\ell_i}\in H'\st_H(2k+1)
\end{align*}
we need $\ell_i\equiv 0 \mod m^{k+1}$ for all $1\leq i\leq m-1$. Thus $(\ell_1,\dots,\ell_{m-1}, \ell_{m})\equiv (0,\dots, 0, -jm^k)$ modulo $m^{k+1}$. However, this is not possible as this vector is not in the image of the matrix $A \mod m^{k+1}$ described in \textcolor{teal}{(\ref{matix A})}. This follows from the fact that the sum of all the coordinates of each row of the matrix $A$ is 0 modulo $m^{k+1}$.

To prove \textcolor{teal}{\textit{(R1)}},\textcolor{teal}{\textit{(R2)}},\textcolor{teal}{\textit{(Q1)}} and \textcolor{teal}{\textit{(Q2)}} first observe the following equality
\begin{align}
\label{align: m/2}
    \psi((ab)^{m/2})&=(b^{-1},a,\dotsc,b^{-1},a).
\end{align}
We begin with the base cases, as before $k=0$, for \textcolor{teal}{\textit{(R1)}} and \textcolor{teal}{\textit{(R2)}}. The argument is similar to the case  \textcolor{teal}{\textit{(P1)}} and \textcolor{teal}{\textit{(P2)}} with the only difference that in this case $ab$ and $\lambda$ have order $m/2$ modulo $H'\st_H(1)$ and $H'\st_H(2)$ respectively. For \textcolor{teal}{\textit{(Q1)}}, it is clear that from \cref{align: m/2} that $(ab)^{m/2}$ belongs to $\st_H(1)$, to prove that $(ab)^{m/2}\not\in H'\st_H(2)$, assume by contradiction that this is not the case. Then  using \cref{align: m/2} and \cref{proposition: branch structure BSV m}, there exists some $c\in \langle [\lambda,a]^{a^r}\mid 0\le r\le m-2\rangle$ such that 
\begin{align*}
    (b^{-1},a,\dots,b^{-1},a)\psi(c)\in H'\st_H(1)\times \stackrel{m}{\cdots}\times H'\st_H(1)
\end{align*}
where $\psi(c)\equiv_{\psi(L_2)}\big((ab)^{\ell_1},\dots, (ab)^{\ell_m}\big)$ with $\ell_i\in \mathbb{Z}$.
Therefore we deduce that there exists $\ell\in \mathbb{Z}$ such that
$$b^{-1}(ab)^\ell\in H'\mathrm{St}_H(1)=\st_H(1).$$
This is a contradiction because $a\equiv b$ modulo $\st_H(1)$ and $b^{2\ell-1}$ is not in $\st_H(1)$. Indeed $b$ has order $m$ (which is even) modulo $\st_H(1)$ and $2\ell-1$ is odd. For \textcolor{teal}{\textit{(Q2)}}, it follows from \cref{align: m/2} and \cref{align: ab order mod H'St(k)} that $\lambda^{m/2}$ belongs to $H'\st_H(2)$. To prove that $\lambda^{m/2}\not\in H'\st_H(3)$, assume by contradiction that this is not the case. Then  using, as before, \cref{align: m/2} and \cref{proposition: branch structure BSV m}, there exists some $c\in \langle [\lambda,a]^{a^r}\mid 0\le r\le m-2\rangle$ such that 
\begin{align*}
    (1,\dots,1,(ab)^{m/2})\psi(c)\in H'\st_H(2)\times \stackrel{m}{\cdots}\times H'\st_H(2),
\end{align*}
where $\psi(c)\equiv_{\psi(L_2)}\big((ab)^{\ell_1},\dots, (ab)^{\ell_m}\big)$ with $\ell_i\in \mathbb{Z}$.
Therefore we deduce that 
$$(ab)^{\ell_i}\in H'\mathrm{St}_H(2)$$
for $1\leq i\leq m-1$, and
$$(ab)^{m/2}(ab)^{\ell_m}\in H'\mathrm{St}_H(2).$$
Using the base case for \textcolor{teal}{\textit{(Q1)}} and \textcolor{teal}{\textit{(R1)}} we deduce that $\ell_i\equiv 0\mod m$ for $1\leq i\leq m-1$ and $\ell_m\equiv -m/2 \mod m$. Thus $(\ell_1,\dots, \ell_{m-1},\ell_m)\equiv (0,\dots, 0, -m/2)$ modulo $m$. However this is not possible as this vector is not in the image of the matrix $A \mod m$. This follows again as before from the sum of all the coordinates on each row of $A$ being 0 modulo $m$.

Then, by induction on $k$, we have that, for $m\geq 4$ even and $1\leq j\leq m/2-1$ or $m/2+1\leq j\leq m-1$, \textcolor{teal}{Equations (}\ref{align: induction ab}\textcolor{teal}{)} and \textcolor{teal}{(\ref{align: induction lambda})} hold, and for $m\ge 2$ even 
\begin{align}\label{align: induction ab even}
(ab)^{m^{k}2^{-1}}\in H'\st_H(2k-1)\setminus H'\st_H(2k)
\end{align}
and
\begin{align}\label{align: induction lambda even}
    \lambda^{m^k2^{-1}}\in H'\st_H(2k)\setminus H'\st_H(2k+1).
\end{align}

Let us prove first \textcolor{teal}{\textit{(R1)}} and \textcolor{teal}{\textit{(Q1)}} in that specified order. We fix $j$ to be such that either $1\leq j\leq m/2-1$ or $m/2+1\leq j\leq m-1$. As with \textcolor{teal}{\textit{(P1)}}, we have that there exists $h_1\in H'$ such that
\begin{align*}
\psi_2(h_1(ab)^{m^{k+1}/2})&\in \big(H'\mathrm{St}_H(2k-1)\times\stackrel{m^2}{\cdots}\times H'\mathrm{St}_H(2k-1)\big)\cap \psi_2(\mathrm{St}_H(2))
\end{align*}
if $m\geq 4$ even and that there exists $h_2\in H'$ such that 
\begin{align*}
    \psi_2(h_2(ab)^{jm^{k}})&\in \big(H'\mathrm{St}_H(2k-2)\times\stackrel{m^2}{\cdots}\times H'\mathrm{St}_H(2k-2)\big)\cap \psi_2(\mathrm{St}_H(2)),
\end{align*}
if $m\ge 2$ even. Using, as before, that $H$ is weakly regular branch over $H'$,  we get 
$$(ab)^{m^{k+1}/2}\in H'\mathrm{St}_H(2k+1)$$
if $m\ge 2$ is even and 
$$(ab)^{jm^{k}}\in H'\mathrm{St}_H(2k)$$
if furthermore $m\geq 4$, from the injectivity of the map $\psi_2$. We now prove first that 
$$(ab)^{jm^{k}}\notin H'\mathrm{St}_H(2k+1)$$
if $m\ge 2$ is even which concludes the proof of \textcolor{teal}{\textit{(R1)}}. As with \textcolor{teal}{\textit{(P1)}}, assume by contradiction that this is not the case. Therefore, there must exist some element $c\in \langle [\lambda,a]^{a^r}\mid 0\le r\le m-2\rangle$ such that 
\begin{align*}
    (\lambda^{jm^{k-1}},\dots, \lambda^{jm^{k-1}})\psi(c)\in H'\st_H(2k)\times \stackrel{m}{\cdots}\times H'\st_H(2k),
\end{align*}
where $\psi(c)\equiv_{\psi(L_2)}\big((ab)^{\ell_1},\dots, (ab)^{\ell_m}\big)$ with $\ell_i\in \mathbb{Z}$.
Therefore we deduce that there exists $\ell\in \mathbb{Z}$ such that
$$(ab)^\ell\lambda^{jm^{k-1}}\in H'\mathrm{St}_H(2k).$$
If $\ell \not\equiv m^k/2$ modulo $m^k$ then
\begin{align*}
    (ab)^\ell\in H'\st_H(2k-1)\setminus H'\st_H(2k)
\end{align*}
leads to a contradiction as in \textcolor{teal}{\textit{(P1)}}. Thus, we may assume that $\ell=m^k/2$. Let $c\in H'$ with $\psi(c)=((ab)^{\varepsilon_1},\dots, (ab)^{\varepsilon_m})$ be such that
\begin{align}\label{contradiction}
    (ab)^{m^k/2}\lambda^{jm^{k-1}}c\in L_2\mathrm{St}_H(2k).
\end{align}
Then applying $\psi$ and using \cref{align: ab order mod H'St(k)} we get
$$(b^{-m^{k-1}},a^{m^{k-1}},\dotsc,b^{-m^{k-1}},a^{m^{k-1}})(1,\dotsc,1,(ab)^{jm^{k-1}})((ab)^{\varepsilon_1},\dots, (ab)^{\varepsilon_m})$$
is in $H'\mathrm{St}_H(2k-1)\times\stackrel{m}{\cdots} \times H'\mathrm{St}_H(2k-1)$. Multiplying the last two coordinates we obtain
$$\lambda^{m^{k-1}}(ab)^{jm^{k-1}}(ab)^{\varepsilon_{m-1}+\varepsilon_{m}}\in H'\mathrm{St}_H(2k-1).$$

Since $\lambda^{m^{k-1}}\in H'\mathrm{St}_H(2k-1)$ by \cref{align: induction lambda}  we deduce that 
\begin{align*}
    (ab)^{jm^{k-1}}(ab)^{\varepsilon_{m-1}+\varepsilon_{m}}=(ab)^{jm^{k-1}+\varepsilon_{m-1}+\varepsilon_{m}} \in H'\mathrm{St}_H(2k-1)
\end{align*}
so, from \cref{align: induction ab}, we see that $jm^{k-1}+\varepsilon_{m-1}+\varepsilon_{m}$ has to be divisible by $m^k/2$. This implies that $\varepsilon_{m-1}+\varepsilon_{m}$ is not divisible by $m^k/2$. Let us write $\delta:=\varepsilon_{m-1}+\varepsilon_m$. Multiplying every other pair of consecutive coordinates we obtain that 
\begin{align*}
    \varepsilon_i+\varepsilon_{i+1}\equiv 0 \mod m^k/2
\end{align*}
for $1\le i\le m-2$. Therefore  the corresponding vector $(\varepsilon_1,\dots,\varepsilon_{m-1},\varepsilon_{m})$ might be written as 
\begin{align*}
    (\varepsilon,-\varepsilon,\varepsilon,-\varepsilon,\dots,\varepsilon,-\varepsilon,\varepsilon,\delta-\varepsilon)
\end{align*}
modulo $m^k/2$ for some $\varepsilon\in \mathbb{Z}/(m^k/2)\mathbb{Z}$.
However, this vector is not in the image of $A$ modulo $m^k/2$ unless $\delta$ is divisible by $m^k/2$, which contradicts the choice of $\delta$.

We now prove that 
$$(ab)^{m^{k+1}/2}\notin H'\mathrm{St}_H(2k+2)$$
which establishes \textcolor{teal}{\textit{(Q1)}}. Arguing by contradiction, assume that this is not the case. Then, there exists some $c\in \langle [\lambda,a]^{a^r}\mid 0\le r\le m-2\rangle$ such that 
\begin{align*}
    (\lambda^{m^{k}/2},\dots, \lambda^{m^{k}/2})\psi(c)\in H'\st_H(2k+1)\times \stackrel{m}{\cdots}\times H'\st_H(2k+1)
\end{align*}
where $\psi(c)\equiv_{\psi(L_2)}\big((ab)^{\ell_1},\dots, (ab)^{\ell_m}\big)$ with $\ell_i\in \mathbb{Z}$.
Therefore we deduce that there exists $\ell\in \mathbb{Z}$ such that
$$(ab)^\ell\lambda^{m^{k}/2}\in H'\mathrm{St}_H(2k+1).$$
If $m=2$ then 
\begin{align*}
    (ab)^\ell\in H'\st_H(2k)\setminus H'\st_H(2k+1)
\end{align*}
and arguing as in \textcolor{teal}{\textit{(P1)}} we arrive at a contradiction. If $m\geq 4$ and $\ell \not\equiv jm^k$ modulo $m^{k+1}$ for $1\leq j\leq  m/2-1$ or $m/2+1\leq j\leq m-1$ then the following
\begin{align*}
    (ab)^\ell\in H'\st_H(2k)\setminus H'\st_H(2k+1)
\end{align*}
leads to a contradiction as in \textcolor{teal}{\textit{(P1)}}. Thus we may assume that $\ell = jm^k$  with $j$ as above. Let $c\in H'$ with $\psi(c)=((ab)^{\varepsilon_1},\dots, (ab)^{\varepsilon_m})$ 
 be such that
$$(ab)^{jm^k}\lambda^{m^k/2}c\in L_2\mathrm{St}_H(2k+1).$$
Then applying $\psi$ and using \textcolor{teal}{Equations (}\ref{align: ab order mod H'St(k)}\textcolor{teal}{)} and \textcolor{teal}{(\ref{align: m/2})} we get
$$(\lambda^{jm^{k-1}},\dotsc,\lambda^{jm^{k-1}})(1,\dotsc,1,(ab)^{m^k/2})((ab)^{\varepsilon_1},\dots, (ab)^{\varepsilon_m})$$
is in $H'\mathrm{St}_H(2k)\times\stackrel{m}{\cdots} \times H'\mathrm{St}_H(2k)$. Thus, in the first $m-1$ coordinates we get 
$$\lambda^{jm^{k-1}}(ab)^{\varepsilon_i}\in H'\mathrm{St}_H(2k).$$
Since $\lambda^{jm^{k-1}}\not\in H'\mathrm{St}_H(2k)$ by \cref{align: induction lambda} we deduce from \cref{align: induction ab even} that $\varepsilon_i\equiv m^k/2$ modulo $m^k$ for all $1\leq i \leq m-1$. Therefore we may assume that $\varepsilon_i= m^k/2$ for all $1\leq i \leq m-1$, which yields that 
$$\lambda^{jm^{k-1}}(ab)^{m^k/2}\in H'\mathrm{St}_H(2k).$$
As seen in \cref{contradiction}, this leads to a contradiction.

To conclude it remains to prove  \textcolor{teal}{\textit{(R2)}} and \textcolor{teal}{\textit{(Q2)}}. This is done very similarly to \textcolor{teal}{\textit{(P2)}}. We begin with \textcolor{teal}{\textit{(R2)}}:
The containment $\lambda^{jm^{k}}\in H'\mathrm{St}_H(2k+1)$ follows from \textcolor{teal}{\textit{(R1)}} and \cref{align: ab order mod H'St(k)}. To prove that $\lambda^{jm^k}\notin H'\mathrm{St}_H(2k+2)$ we assume by contradiction that this is not the case, and, as before, using \cref{align: ab order mod H'St(k)} and \cref{proposition: branch structure BSV m}, there exists some $c\in \langle [\lambda,a]^{a^r}\mid 0\le r\le m-2\rangle$ such that 
\begin{align*}
    (1,\dots,1, (ab)^{jm^{k}})\psi(c)\in H'\st_{H}(2k+1)\times \stackrel{m}{\cdots}\times H'\st_H(2k+1),
\end{align*}
where $\psi(c)\equiv_{\psi(L_2)}\big((ab)^{\ell_1},\dots, (ab)^{\ell_m}\big)$ with $\ell_i\in \mathbb{Z}$. From \textcolor{teal}{\textit{(Q1)}} and \textcolor{teal}{\textit{(R1)}}, we deduce that in order for 
\begin{align*}
  (ab)^{jm^{k}}  (ab)^{\ell_m}=(ab)^{jm^{k}+\ell_m}\in H'\st_H(2k+1)
\end{align*}
to hold $jm^{k}+\ell_m$ has to be divisible by $m^{k+1}/2$, i.e. $jm^{k}+\ell_m\equiv 0 \mod m^{k+1}/2$. Similarly  for 
\begin{align*}
  (ab)^{\ell_i}\in H'\st_H(2k+1)
\end{align*}
we need $\ell_i\equiv 0 \mod m^{k+1}/2$ for all $1\leq i \leq m-1$. Therefore we see that $(\ell_1,\dots,\ell_{m-1}, \ell_{m})\equiv(0,\dots, 0, -jm^k)$ modulo $m^{k+1}/2$. However this is not in the image of the matrix $A \mod m^{k+1}/2$, which yields a contradiction.

For \textcolor{teal}{\textit{(Q2)}} the containment $\lambda^{m^{k+1}/2}\in H'\mathrm{St}_H(2k+2)$ follows from \textcolor{teal}{\textit{(Q1)}} and \cref{align: ab order mod H'St(k)}. To prove that $\lambda^{m^{k+1}/2}\notin H'\mathrm{St}_H(2k+3)$ we assume by contradiction that this is not the case, therefore, as before, using \cref{align: ab order mod H'St(k)} and \cref{proposition: branch structure BSV m}, there exists some $c\in \langle [\lambda,a]^{a^r}\mid 0\le r\le m-2\rangle$ such that 
\begin{align*}
    (1,\ldots,1, (ab)^{m^{k+1}/2})\psi(c)\in H'\st_{H}(2k+2)\times \stackrel{m}{\cdots}\times H'\st_H(2k+2)
\end{align*}
where $\psi(c)\equiv_{\psi(L_2)}\big((ab)^{\ell_1},\dots, (ab)^{\ell_m}\big)$ with $\ell_i\in \mathbb{Z}$. From \textcolor{teal}{\textit{(Q1)}} and \textcolor{teal}{\textit{(R1)}}, we deduce that in order to 
\begin{align*}
  (ab)^{m^{k+1}/2}  (ab)^{\ell_m}=(ab)^{m^{k+1}/2+\ell_m}\in H'\st_H(2k+2)
\end{align*}
to hold $m^{k+1}/2+\ell_m$ has to be divisible by $m^{k+1}$, in other words $m^{k+1}/2+\ell_m\equiv 0 \mod m^{k+1}$. Similarly  for 
\begin{align*}
  (ab)^{\ell_i}\in H'\st_H(2k+2)
\end{align*}
we need $\ell_i\equiv 0 \mod m^{k+1}$ for all $1\leq i \leq m-1$. As before we see that $(\ell_1,\ldots,\ell_{m-1}, \ell_{m})\equiv (0,\cdots, 0, -{m^{k+1}/2})$ modulo $m^{k+1}$, but this is not in the image of the matrix $A \mod m^{k+1}$. Therefore we have a contradiction in this case too.
\end{proof}

Let $n\ge 1$ and we define $k_n:=\lfloor n/2\rfloor$ and $e_n:=n-2k_n$. Then we can write $n=2k_n+e_n$. We note the equalities 
$$k_{n-1}=k_n-1+e_n,\quad e_{n-1}=1-e_n,\quad k_{n+1}=k_n+e_n,\quad e_{n+1}=1-e_n.$$ 
Recall that in \cref{align: definition of tau} we have defined the parameter $\tau$ as
\begin{align*}
    \tau:=\begin{cases}
        0& \text{if }m \text{ is odd},\\
        1& \text{if }m \text{ is even}.
    \end{cases}
\end{align*}

\begin{proposition}
\label{proposition: order of H' over LSt(n)}
For $n\ge 2$, the logarithmic orders of the quotients $H'/L_2\mathrm{St}_{H'}(n)$ are
\begin{align*}
    \log_m|H':L_2\mathrm{St}_{H'}(n)|=(m-1)(k_n-e_{n-1}\tau\log_m 2).
\end{align*}
\end{proposition}
\begin{proof}
In \cref{Lemma: quotient direct product of finite cyclic groups} we have shown that $H'/L_2\mathrm{St}_{H'}(n)$ is a direct product of $m-1$ cyclic groups of order $\ell$, where $\ell$ is the order of $ab$ modulo $H'\st_H(n-1)$. By \cref{proposition: orders of elements in abelianization m} for $m$ odd we have that modulo $H'\st_{H}(2k-1)$ and modulo $H'\st_{H}(2k)$ the element $ab$ has order $m^k$. Therefore the logarithmic order of $ab$ modulo $H'\mathrm{St}_H(n)$ is 
\begin{align*}
    k_n+e_n,
\end{align*}
for $n\ge 1$.
For $m>2$ even, modulo $H'\st_{H}(2k)$ the element $ab$ has order $m^k$, while modulo $H'\st_{H}(2k+1)$ the order of $ab$ is $m^{k+1}2^{-1}$. It follows that the logarithmic order of $ab$ modulo $H'\mathrm{St}_H(n)$ is 
\begin{align*}
    k_n+e_n-e_n\log_{m}2,
\end{align*}
for $n\ge 1$.
Lastly, for $m=2$ modulo $H'\st_{H}(2k)$ and modulo $H'\st_{H}(2k+1)$ the element $ab$ has order $m^k$. Hence, the logarithmic order of $ab$ modulo $H'\mathrm{St}_H(n)$ is
\begin{align*}
    k_n+e_n-e_n\log_{2}2=k_n,
\end{align*}
for $n\ge 1$. Therefore, for $m\geq 2$ and any $n\geq 1$ the logarithmic order of $ab$ modulo $H'\mathrm{St}_H(n)$ is 
$$k_n+e_n-e_n\tau\log_m 2,$$
and thus  
$$\log_m|H':L_2\st_{H'}(n)|=(m-1)(k_{n-1}+e_{n-1}-e_{n-1}\tau\log_m 2).$$
Hence the result follows from the equality $k_{n}=k_{n-1}+e_{n-1}$.
\end{proof}

Recall the formula for the sequence $\{s_n\}_{n\ge 1}$ of a  branching groups given in \cref{align: sn for branching}. 

\begin{lemma}
\label{lemma: sn para H'}
The sequence $\{s_n\}_{n\ge 1}$ for $H'$ is given by
\begin{align*}
    s_n=(m-1)\big((2e_n-1)\tau\log_m 2-e_{n}\big).
\end{align*}
\end{lemma}
\begin{proof}
Since $H'$ is branching, the result follows from \cref{proposition: order of H' over LSt(n)} and both the equalities  $k_{n+1}=k_n+e_n$ and $ e_{n-1}=1-e_n$ as
\begin{align*}
    s_n&=-\log_m|L_2\mathrm{St}_{H'}(n):L_2\mathrm{St}_{H'}(n+1)|\\
    &=-\log_m|H':L_2\mathrm{St}_{H'}(n+1)|+\log_m|H':L_2\mathrm{St}_{H'}(n)|\\
    &=-(m-1)(k_{n+1}-e_{n}\tau\log_{m}2)+(m-1)(k_n-e_{n-1}\tau\log_{m}2)\\
    &=(m-1)(k_n-k_{n+1}+(e_{n}-e_{n-1})\tau\log_{m}2)\\
    &=(m-1)(-e_n+(2e_{n}-1)\tau\log_{m}2).\qedhere
\end{align*}
\end{proof}

\begin{proof}[Proof of \cref{Theorem: Hausdorff dimension and congruence orders for m}]
By \cref{proposition: G and G' same hausdorff dimension}, we get $\mathrm{hdim}_{\Gamma_m}(\overline{H})=\mathrm{hdim}_{\Gamma_m}(\overline{H'})$. Since the group $H'$ is branching, using \cref{lemma: sn para H'}, we may compute its Hausdorff dimension as 
\begin{align*}
    \mathrm{hdim}_{\Gamma_m}(\overline{H'})&=\log_m|H':\mathrm{St}_{H'}(1)|-S_{H'}(1/m)=-\sum_{n=1}^\infty\frac{s_n}{m^n}\qquad\qquad\quad\\
    &=-\sum_{n=1}^\infty\frac{(m-1)\big((2e_n-1)\tau\log_{m}2-e_n\big)}{m^n}\\
\end{align*}
\begin{align*}
    \qquad\qquad\qquad\quad&=-(m-1)\left(\sum_{n=1}^\infty\frac{(2e_n-1)\tau\log_{m}2-e_n}{m^n}\right)\\
    &=(m-1)\left(\sum_{n=1}^\infty\frac{e_n}{m^n}-\tau(\log_{m}2)\left(2\sum_{n=1}^\infty\frac{e_n}{m^n}-\sum_{n=1}^\infty\frac{1}{m^n}\right)\right)\\
    &=(m-1)\left(\sum_{n=1}^\infty m^{1-2n}-\tau(\log_{m}2)\left(2\sum_{n=1}^\infty m^{1-2n}-\sum_{n=1}^\infty m^{-n}\right)\right)\\
    &=(m-1)\left(\frac{m}{m^2-1}-\tau(\log_{m}2)\left(2\frac{m}{m^2-1}-\frac{1}{m-1}\right)\right)\\
    &=\frac{m-\tau(m-1)\log_m 2}{m+1}.\qedhere
\end{align*} 
\end{proof}

Note that for $m$ even and not a power of 2, the real number $\log_m 2$ is irrational (otherwise $\log_m 2=a/b$ being rational would imply that $m^a=2^b$ which is impossible as $m$ contains a prime factor different from 2 by assumption). Furthermore, as $m^{\log_m 2}=2$ is rational the irrational number $\log_m 2$ must be transcendental by the Gelfond-Schneider Theorem.

\bibliographystyle{unsrt}

\typeout{get arXiv to do 4 passes: Label(s) may have changed. Rerun}

\end{document}